\documentclass{amsart}
\usepackage[utf8]{inputenc}
\usepackage{amssymb}

\usepackage{fancyhdr}
\usepackage{hyperref}

\usepackage{amsfonts}
\usepackage{amsmath}
\usepackage{amsthm}
\DeclareMathOperator{\Sym}{Sym}

\usepackage{tikz}

\usepackage{graphicx}

\usepackage{caption}

\fancypagestyle{alim}{\fancyhf{}\fancyfoot[L]{
    \noindent\rule{6cm}{0.4pt}
    The author would like to thank the Math Department of the University of North Texas’s Incubator project for its encouragement and guidance.
    }}

\newtheorem{theorem}{Theorem}[section]

\newtheorem{definition}{Definition}[section]

\title{ The completeness of 2D Rubik's Shapes }
\author{Skylar Werner}
\address{Skylar Werner, Georg-August-Universität Göttingen, Department of Mathematics, Wilhelmsplatz 1, 37073 Göttingen, Germany}
\email{skylar.werner@stud.uni-goettingen.de}

\begin{document}

\maketitle

\thispagestyle{alim}

% \section*{Abstract}

\begin{abstract}
    The Rubik's cube was invented in 1974 by Erno Rubik, who had no idea of the incredible popularity and mathematical fascinations his toy would bring. Through the years of study on the mathematical properties of the cube, the Rubik's Cube group was introduced to represent all possible moves one could perform on the cube. In this paper, we define a planar analogue to the Rubik's cube, which we dub the Rubik's Square, and prove that the Rubik's square is \textit{complete} in the sense that given any two configurations there is a sequence of moves which changes one to the other. The Rubik's cube does not have this property. We then abstract the concept of the Rubik's Square to a Rubik's Shape and analyse the \textit{completeness} in this more general setting.

\end{abstract}

\section{Background and Overview}
The Rubik's cube was introduced in the 1970's, and consists of a 3x3x3 configuration of small cubes creating the larger cube.  The center piece on each of the six sides is stationary, with the other pieces rotating around it.  As such there are 48 movable pieces to consider.

\subsection{}

One widely accepted standard for describing the cube and its moves is as follows:\\

\begin{minipage}[c]{0.45\textwidth}
    \begin{itemize}
        \item F - Front Face
        \item B - Back Face
        \item L - Left Face
        \item R - Right Face
        \item U - Up (Top) Face
        \item D - Down (Bottom) Face
    \end{itemize}
\end{minipage}
\begin{minipage}[c]{0.45\textwidth}
    \centering
    \includegraphics[scale=0.5]{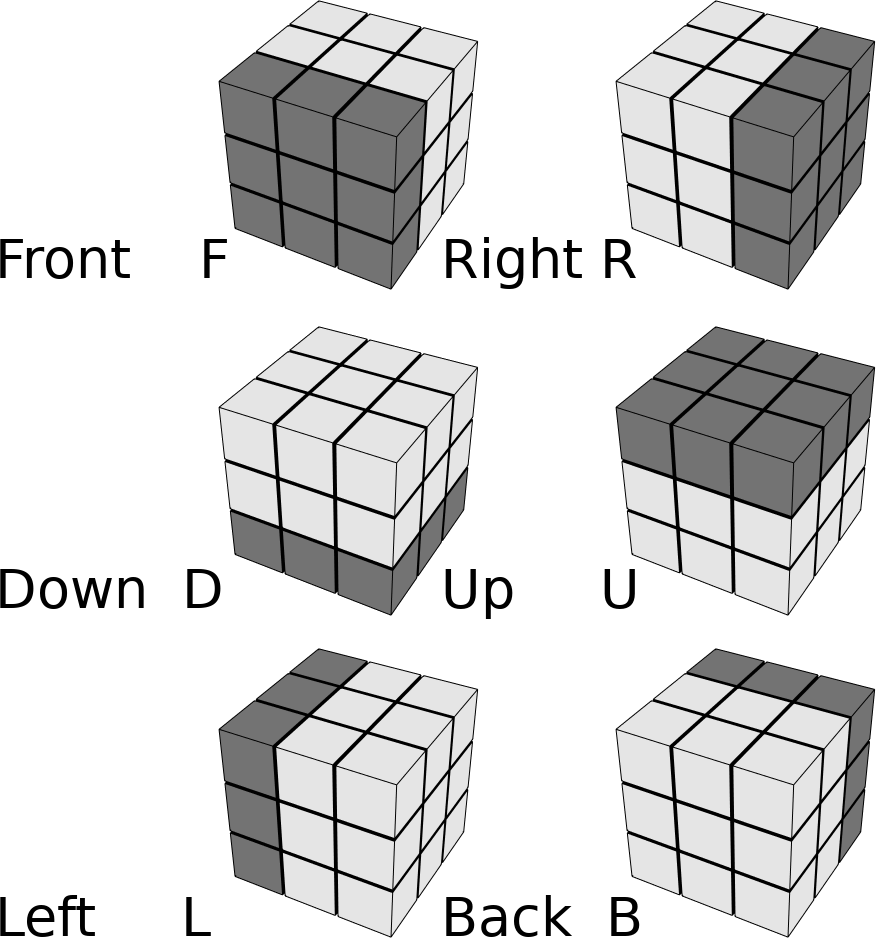}
    \captionof{figure}{Rubik's Cube moves}
\end{minipage}

The letters {F,B,L,R,U,D} are used not only to describe the faces of the cube, but also their motions, i.e. F would be a $90^{\circ}$ clockwise rotation of the front face.  These six basic moves generate all possible configurations of the standard Rubik's cube.

\subsection{}

There exists extensive \cite{Bandelow} \cite{GodsNumber} literature concerning the mathematical properties of the Rubik's cube.  In this paper we wanted to explore a 2-dimensional analog of the cube, as well as explore how that construction could be expanded to a more general case.
\\
In section 3 we introduce the $2 \times 2$ Rubik's square, giving its properties and algorithms for solving any possible state.
\\
Section 4 gives a simple generic case for 2-dimensional connected polygons called Rubik's shapes. It explores configurations with polygons connected by a single side, showing their relation to familiar concepts in group theory.  This section also presents an algorithm for this simple case when at least one of the polygons is even-sided.

\section{2 x 2 Rubik's square}

We start our study on the most natural 2-dimensional analog to the Rubik's cube, the $2 \times 2$ Rubik's square. We first define what a $2 \times 2$ Rubik's square is, the initial states, and when it is complete. Then we define a set of algorithms that will become the building blocks proving that every $2 \times 2$ Rubik's square is complete.

\begin{definition}
    The $2 \times 2$ Rubik's square is a graph $G$, with 4-cycles $\{C_1, \dots, C_4\}$. The shape is counstructed out of $4$ squares arranged to create a bigger square where the rotation of each smaller square is defined by $C_i$ for $i = 1, \dots, 4$.
\end{definition}

Each edge takes upon $1$ of $4$ colors where the colors can only be used $3$ times, no more or less. The edges will be denoted by numbers and the color they take on with a subscript, i.e. $1_r$ is edge 1 with color red. Note the numbering of the edges are for notational purpose only and the coloring is what truly matters.

\begin{definition}
    The initial state of a $2 \times 2$ Rubik's square is where the set of edges adjacent to a degree $3$ vertice is of the same color.
\end{definition}

Thus there are $4! = 24$ different initial states. In this paper, we will show that the initial state can be reduced to a single case. In Figure $2$ below, we will assume this is our initial state for this paper.

\begin{minipage}{.45\linewidth}
    \begin{itemize}
        \item $C_1 = (1\,4\,6\,3)$
        \item $C_2 = (2\,5\,7\,4)$
        \item $C_3 = (6\,9\,11\,8)$
        \item $C_4 = (7\,10\,12\,9)$
    \end{itemize}
\end{minipage}\hfill
\begin{minipage}{.45\linewidth}
    \begin{tikzpicture}
        % \draw[white] (0,0) -- (8,0);
        \draw[ultra thick] (4,2) -- node[above] {$1_r$} (6,2);
        \draw[ultra thick] (6,2) -- node[above] {$2_r$} (8,2);
        \draw[ultra thick] (6,2) -- node[left] {$4_r$} (6,0);
        \draw[ultra thick] (4,2) -- node[left] {$3_b$} (4,0);
        \draw[ultra thick] (4,0) -- node[left] {$8_b$} (4,-2);
        \draw[ultra thick] (4,0) -- node[above] {$6_b$} (6,0);
        \draw[ultra thick] (8,2) -- node[left] {$5_w$} ( 8,0);
        \draw[ultra thick] (8,0) -- node[left] {$10_w$} (8,-2);
        \draw[ultra thick] (6,0) -- node[above] {$7_w$} (8,0);
        \draw[ultra thick] (4,-2) -- node[above] {$11_g$} (6,-2);
        \draw[ultra thick] (6,-2) -- node[left] {$9_g$} (6,0);
        \draw[ultra thick] (6,-2) -- node[above] {$12_g$}(8,-2);
    \end{tikzpicture}
    \centering
    
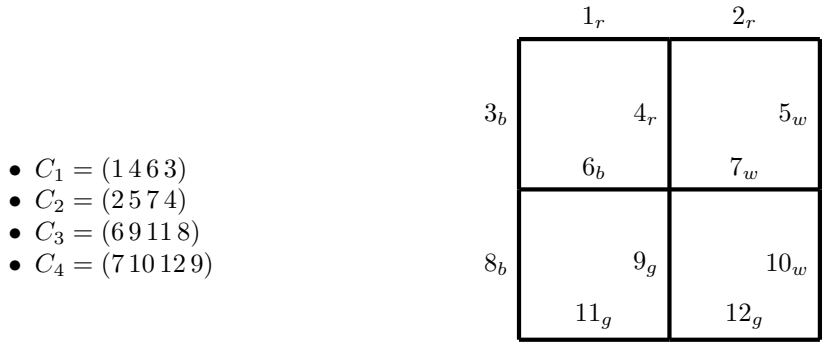
\captionof{figure}{Rubik's Square inital state}
\end{minipage}

\begin{definition}
    A $2 \times 2$ Rubik's square is \textit{complete} if given any two configurations there exists a sequence of rotations which transforms one configuration to the other.
\end{definition}

\subsection{Color Group Moves}

The Rubik's Cube can be physically rotated. With this observation in mind, allowing the Rubik's square to be rotatable was considered.

\begin{definition}
    A Color Group is the set of edges adjacent to a degree 3 vertices, e.g $\{1,2,4\}$.

    A Color Group Move is an algorithm that switches two color groups and is denoted by $G_{IJ}$ where $I, J \in \{T, L, R, B\}$ are representing the top, left, right, or bottom color group. For example, $G_{TL}$ is the algorithm switching the top and left color groups.
\end{definition}

The main idea of this section is to construct four basic Color Group Moves (see $1$ through $4$ below) which are then composed to abtain all other Color Group Moves. Notice that we are only interested in the color and not the ordering of the edges. For example, in figures $3$ and $4$ we apply $G_{TL}$ transforming an initial state to an another intial state.

\begin{enumerate}
    \item $G_{TL} = M_1M_3M_1M_2M_2M_2M_3M_3M_3M_1M_2$
    \item $G_{TR} = M_2M_1M_2M_4M_2M_4M_2M_4M_2M_1M_1M_1M_2M_4M_4M_4M_2M_4M_4M_4M_2M_2M_2$
    \item $G_{LB} = M_3M_4M_3M_1M_3M_1M_3M_1M_3M_3M_1M_3M_1M_1M_3M_3M_3M_4M_4M_4$
    \item $G_{RB} = M_4M_2M_4M_3M_3M_3M_2M_4M_2M_2M_4M_4M_4M_2M_4M_3$
    \item $G_{TB} = G_{TL}G_{LB}G_{TL}$
    \item $G_{LR} = G_{TL}G_{TR}G_{TL}$
\end{enumerate}
\begin{minipage}{.45\linewidth}
    \begin{tikzpicture}
        \draw[white] (0,0) -- (8,0);
        \draw[ultra thick] (1,2) -- node[above] {$1_r$} (3,2);
        \draw[ultra thick] (3,2) -- node[above] {$2_r$} (5,2);
        \draw[ultra thick] (3,2) -- node[left] {$4_r$} (3,0);
        \draw[ultra thick] (1,2) -- node[left] {$3_b$} (1,0);
        \draw[ultra thick] (1,0) -- node[left] {$8_b$} (1,-2);
        \draw[ultra thick] (1,0) -- node[above] {$6_b$} (3,0);
        \draw[ultra thick] (5,2) -- node[left] {$5_w$} (5,0);
        \draw[ultra thick] (5,0) -- node[left] {$10_w$} (5,-2);
        \draw[ultra thick] (3,0) -- node[above] {$7_w$} (5,0);
        \draw[ultra thick] (1,-2) -- node[above] {$11_g$} (3,-2);
        \draw[ultra thick] (3,-2) -- node[left] {$9_g$} (3,0);
        \draw[ultra thick] (3,-2) -- node[above] {$12_g$} (5,-2);
    \end{tikzpicture}
    \centering
    
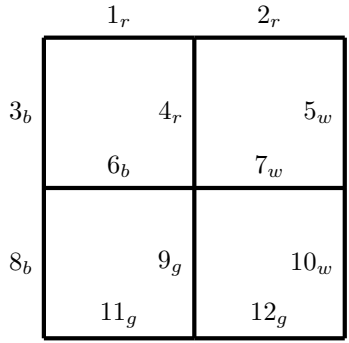
\captionof{figure}{Rubik's Square inital state}
\end{minipage}\hfill
\begin{minipage}{.45\linewidth}
    \begin{tikzpicture}
        \draw[ultra thick] (4,2) -- node[above] {$8_b$} (6,2);
        \draw[ultra thick] (6,2) -- node[above] {$6_b$} (8,2);
        \draw[ultra thick] (6,2) -- node[left] {$3_b$} (6,0);
        \draw[ultra thick] (4,2) -- node[left] {$4_r$} (4,0);
        \draw[ultra thick] (4,0) -- node[left] {$1_r$} (4,-2);
        \draw[ultra thick] (4,0) -- node[above] {$2_r$} (6,0);
        \draw[ultra thick] (8,2) -- node[left] {$5_w$} ( 8,0);
        \draw[ultra thick] (8,0) -- node[left] {$10_w$} (8,-2);
        \draw[ultra thick] (6,0) -- node[above] {$7_w$} (8,0);
        \draw[ultra thick] (4,-2) -- node[above] {$11_g$} (6,-2);
        \draw[ultra thick] (6,-2) -- node[left] {$9_g$} (6,0);
        \draw[ultra thick] (6,-2) -- node[above] {$12_g$}(8,-2);
    \end{tikzpicture}
    \centering
    
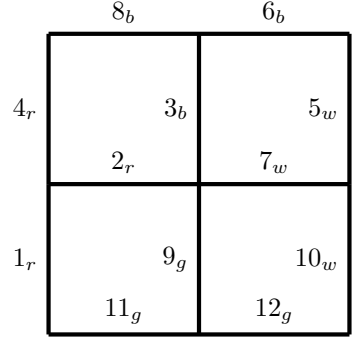
\captionof{figure}{Initial state after applying $G_{TL}$}
\end{minipage}\\

By symmetry, we could create similar algorithms for each Color Group Move where they would mimick each other. These algorithms move the edges around differently, an observation that will be utilized in the next section.

\subsection{Line Moves}

The most powerful and useful algorithms for the Rubik's square are the ones switching any two sides. In the definition below, we attempt to define a similar algorithms to fit this set up.
%  We will show that there is such an algorithm and composing it with Color Group Moves and Lime Moves will give every Line Moves.

\begin{definition}
    A Line Move is an algorithm that switches any two edges.
\end{definition}

We denote a Line Move switching edges $X$ and $Y$ by $L_{XY}$. For example, $L_{13}$ is the algorithm switching edge $1$ and $3$ shown below.

\begin{itemize}
    \item $L_{1\,3} = M_2M_1M_2M_4M_4M_4M_3M_3M_1M_1M_1M_3M_1M_3M_3M_4M_2M_2M_1M_2$
\end{itemize}
\begin{minipage}{.45\linewidth}
    \begin{tikzpicture}
        \draw[white] (0,0) -- (8,0);
        \draw[ultra thick] (1,2) -- node[above] {$1_r$} (3,2);
        \draw[ultra thick] (3,2) -- node[above] {$2_r$} (5,2);
        \draw[ultra thick] (3,2) -- node[left] {$4_r$} (3,0);
        \draw[ultra thick] (1,2) -- node[left] {$3_b$} (1,0);
        \draw[ultra thick] (1,0) -- node[left] {$8_b$} (1,-2);
        \draw[ultra thick] (1,0) -- node[above] {$6_b$} (3,0);
        \draw[ultra thick] (5,2) -- node[left] {$5_w$} (5,0);
        \draw[ultra thick] (5,0) -- node[left] {$10_w$} (5,-2);
        \draw[ultra thick] (3,0) -- node[above] {$7_w$} (5,0);
        \draw[ultra thick] (1,-2) -- node[above] {$11_g$} (3,-2);
        \draw[ultra thick] (3,-2) -- node[left] {$9_g$} (3,0);
        \draw[ultra thick] (3,-2) -- node[above] {$12_g$} (5,-2);
    \end{tikzpicture}
    \centering
    
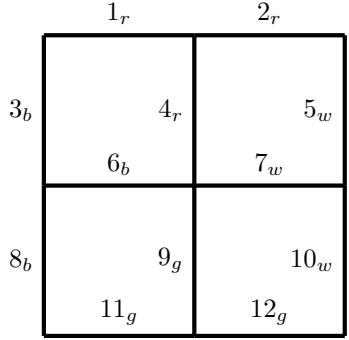
\captionof{figure}{Rubik's Square inital state}
\end{minipage}\hfill
\begin{minipage}{.45\linewidth}
    \begin{tikzpicture}
        \draw[ultra thick] (1,2) -- node[above] {$3_b$} (3,2);
        \draw[ultra thick] (3,2) -- node[above] {$2_r$} (5,2);
        \draw[ultra thick] (3,2) -- node[left] {$4_r$} (3,0);
        \draw[ultra thick] (1,2) -- node[left] {$1_r$} (1,0);
        \draw[ultra thick] (1,0) -- node[left] {$8_b$} (1,-2);
        \draw[ultra thick] (1,0) -- node[above] {$6_b$} (3,0);
        \draw[ultra thick] (5,2) -- node[left] {$5_w$} (5,0);
        \draw[ultra thick] (5,0) -- node[left] {$10_w$} (5,-2);
        \draw[ultra thick] (3,0) -- node[above] {$7_w$} (5,0);
        \draw[ultra thick] (1,-2) -- node[above] {$11_g$} (3,-2);
        \draw[ultra thick] (3,-2) -- node[left] {$9_g$} (3,0);
        \draw[ultra thick] (3,-2) -- node[above] {$12_g$} (5,-2);
    \end{tikzpicture}
    \centering
    
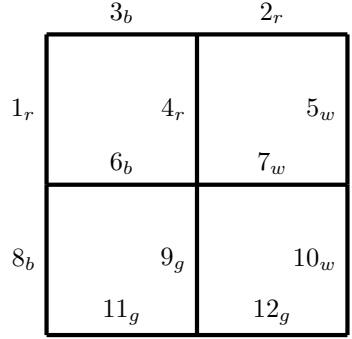
\captionof{figure}{Initial state after applying $L_{1\,3}$}
\end{minipage}\\

By symmetry, there are Line Moves for $L_{2\,5}, L_{8\,11},$ and $L_{10\,12}$. Using how the Color Group Moves move the edges around, we will define some of the different Line Moves. We leave the ones not written for the reader.

\begin{itemize}
    \item $L_{4\,8} = G_{TL}L_{1\,3}G_{TL}^{-1}$
    \item $L_{5\,6} = G_{TL}L_{2\,5}G_{TL}^{-1}$
    \item $L_{1\,11} = G_{TB}L_{8\,11}G_{TB}^{-1}$
    \item $L_{3\,11} = L_{1\,3}L_{1\,11}L_{1\,3}$
    \item $L_{4\,7} = G_{TR}L_{2\,5}G_{TR}^{-1}$
\end{itemize}

It can be checked that all Line Moves can be defined and thus every $2 \times 2$ Rubik's square is complete.

\section{2D Rubik's shape with one connected edge given one shape has an even side}

% Following the analysis of the $n \times n$ Rubik's square, we continued the study on more general Rubik's shapes.

In this section, we attempt to analyse the completeness of more general Rubik's shapes. Conceptually, the construction of a Rubik's shape starts with any polygon followed by inductively constructing new polygons from the sides of the old polygons, see figure $7$ below. We do not allow the new polygon to be made from more than one polygon.

\tikzset{every picture/.style={line width=0.75pt}}

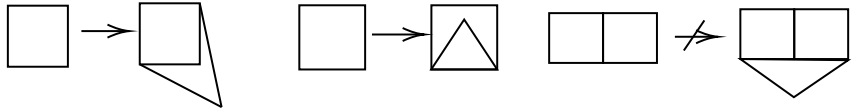
\begin{figure}[!ht]
    \centering
    \begin{tikzpicture}[x=0.75pt,y=0.75pt,yscale=-1,xscale=1]
        %uncomment if require: \path (0,489); %set diagram left start at 0, and has height of 489

        %Shape: Rectangle [id:dp7432087033119543] 
        \draw   (485.43,33) -- (512.61,33) -- (512.61,58.13) -- (485.43,58.13) -- cycle ;
        %Straight Lines [id:da041733192385578155] 
        \draw    (153.09,44.03) -- (175.18,44.03) ;
        \draw [shift={(177.18,44.03)}, rotate = 180] [color={rgb, 255:red, 0; green, 0; blue, 0 }  ][line width=0.75]    (10.93,-3.29) .. controls (6.95,-1.4) and (3.31,-0.3) .. (0,0) .. controls (3.31,0.3) and (6.95,1.4) .. (10.93,3.29)   ;
        %Shape: Rectangle [id:dp5348649277528039] 
        \draw   (182.52,30) -- (212.8,30) -- (212.8,60.82) -- (182.52,60.82) -- cycle ;
        %Straight Lines [id:da9296734195867984] 
        \draw    (182.52,60.82) -- (223.8,82.4) ;
        %Straight Lines [id:da005247428457985315] 
        \draw    (212.8,30) -- (223.8,82.4) ;
        %Shape: Rectangle [id:dp8708807763651234] 
        \draw   (263,31) -- (296.19,31) -- (296.19,63.4) -- (263,63.4) -- cycle ;
        %Straight Lines [id:da12804329646991652] 
        \draw    (299.64,45.75) -- (324.05,45.75) ;
        \draw [shift={(326.05,45.75)}, rotate = 180] [color={rgb, 255:red, 0; green, 0; blue, 0 }  ][line width=0.75]    (10.93,-3.29) .. controls (6.95,-1.4) and (3.31,-0.3) .. (0,0) .. controls (3.31,0.3) and (6.95,1.4) .. (10.93,3.29)   ;
        %Shape: Rectangle [id:dp10158418314753836] 
        \draw   (329.61,31) -- (362.8,31) -- (362.8,63.4) -- (329.61,63.4) -- cycle ;
        %Shape: Triangle [id:dp3419162387952932] 
        \draw   (346.15,38.19) -- (362.8,63.4) -- (329.5,63.4) -- cycle ;
        %Shape: Rectangle [id:dp7904500566000126] 
        \draw   (116,31.2) -- (146.28,31.2) -- (146.28,62.02) -- (116,62.02) -- cycle ;
        %Shape: Rectangle [id:dp5529772826423991] 
        \draw   (512.61,33) -- (539.8,33) -- (539.8,58.13) -- (512.61,58.13) -- cycle ;
        %Shape: Triangle [id:dp27488571451991106] 
        \draw   (512.43,77.4) -- (485.43,58.13) -- (539.79,58.71) -- cycle ;
        %Shape: Rectangle [id:dp028616888067886093] 
        \draw   (389,34.96) -- (416.19,34.96) -- (416.19,60.09) -- (389,60.09) -- cycle ;
        %Shape: Rectangle [id:dp6256676910136123] 
        \draw   (416.19,34.96) -- (443.37,34.96) -- (443.37,60.09) -- (416.19,60.09) -- cycle ;
        %Straight Lines [id:da2655759393206305] 
        \draw    (452.41,46.89) -- (472.04,46.89) ;
        \draw [shift={(474.04,46.89)}, rotate = 180] [color={rgb, 255:red, 0; green, 0; blue, 0 }  ][line width=0.75]    (10.93,-3.29) .. controls (6.95,-1.4) and (3.31,-0.3) .. (0,0) .. controls (3.31,0.3) and (6.95,1.4) .. (10.93,3.29)   ;
        %Straight Lines [id:da25848213080043303] 
        \draw    (467.29,38.63) -- (456.93,53.83) ;

        % Text Node
        % \draw (128,112) node [anchor=north west][inner sep=0.75pt]   [align=left] {Rubik Shape};
        % Text Node
        % \draw (267,110) node [anchor=north west][inner sep=0.75pt]   [align=left] {Rubik Shape};
        % Text Node
        % \draw (395,108) node [anchor=north west][inner sep=0.75pt]   [align=left] {Not a Rubik Shape};
    \end{tikzpicture}
    \centering
    \captionof{figure}{The first two are valid, the third isn't}
\end{figure}

\begin{definition}\label{RubikShapeDef}
    A \emph{2d-Rubik's shape} is a graph $G$, with a set of distinct cycles $\{C_1, \dots, C_n\}$ of $G$  which is built up by induction as follows:

    \begin{enumerate}
        \item The $n$-cycles for $n \geq 3$ are all 2d-Rubik's shapes.
        \item Given a 2d-Rubik's shape $G$ with distinguished cycles $\{C_1, \dots C_n\}$, if we let $P$ be a path in $G$ with at least two vertices so that $E(P) \cap E(C_i) \cap E(C_j) = \emptyset$ for any distinct $i, j$, and let $G' = G \cup \{v_1, \dots v_k\}$ with $v_i$ is connected to $v_{i+1}$ and $v_1$ and $v_k$ are connected to the different endpoints of $P$ by an edge, then $G'$, together with the distinguished cycles $\{C_1, \dots C_n\} \cup \{P \cup \{v_1, \dots v_k\}\}$ is a 2d-Rubik's shape.
    \end{enumerate}
\end{definition}

% The intuition behind this definition is that a complete 2d-Rubik's shape can be returned by use of the allowed permutations $\sigma_1, \dots \sigma_n$ from any initial configuration of the edges to the starting configuration.

The definition below captures the abstract notion of completeness for the $2$d-Rubik's shape.

\begin{definition}
    A 2d-Rubik's shape $G, \{C_1, \dots C_n\}$ is \emph{complete} if the subgroup of $H \leq \Sym(E(G))$ generated by the permutations $\sigma_1, \dots, \sigma_n$ which respectively rotate the edges of the cycles $C_1, \dots, C_n$ is isomorphic to $\Sym(E(G))$.
\end{definition}

We consider the case of Rubik's shapes where every cycle shares only zero or one edges to every other cycle, i.e. $|E(C_i) \cap E(C_j)| = \{0, 1\}$ for all $i,j \in \mathbb{N}$ such that $i \neq j$.

% We start our analysis and only consider this case in this paper on the subset of Rubik's shapes where every cycle shares only zero or one edges to every other cycle, i.e. $|E(C_i) \cap E(C_j)| = \{0, 1\}$ for all $i,j \in \mathbb{Z}$ such that $i \neq j$. Interestingly, when a 2d-Rubik's shape contains at least one cycle of even length, then it is complete. We prove this statement by induction

\begin{theorem}
    % Suppose that $G$ is a 2d-Rubik's shape with distinguished cycles ${C_1, C_2}$ and that the length of at least one of $C_1, C_2$ is even and $|E(C_1) \cap E(C_2)| = 1$.  Then $G$ is a complete 2d-Rubik's shape.

    Suppose that $G$ is a complete 2d-Rubik's shape with distinct cycles $C_1, C_2, \dots, C_n$ where at least one cycle is even. Then adding a polygon following the construction from Defintion \ref*{RubikShapeDef} results in a new complete 2d-Rubik's shape.
    % Then adding new vertices defined by \textbf{Definition 4.1} with path $P$ in $G$ having only two vertices results in a new complete 2d-Rubik's shape.
\end{theorem}
\begin{proof}
    Base Case: Suppose we have a Rubik's shape with cycles $C_1$ and $C_2$. Without loss of generality, suppose that $C_1$ is even.
    Label the edges of $C_1$ by $(0, 1, \dots, k)$ and the edges of $C_2$ by $(0, k+\ell, k+\ell-1, \dots, k+2, k+1)$, so that if we identify $\sigma_1, \sigma_2$ with the corresponding elements of $\Sym(k+\ell+1)$ then $\sigma_1$ is the cycle $(0, 1, \dots, k)$ and $\sigma_2$ is the cycle $(0, k+\ell, k+\ell-1, \dots, k+2, k+1)$

    We claim that if $0\leq j-1 < j \leq k$ then the product,
    \[ \sigma_1^{j+1}\left(\sigma_2^{-1}\sigma_1\sigma_2\sigma_1\right)^{\frac{k-2}{2}}\sigma_2^{-1}\sigma_1^2\sigma_2\sigma_1^{k+1-j} \]
    will transpose the edges $j-1$ and $j$.

    % \vspace{5cm}
    % \phantom{hello}

    \tikzset{every picture/.style={line width=0.75pt}} %set default line width to 0.75pt        

    \begin{tikzpicture}[x=0.75pt,y=0.75pt,yscale=-1,xscale=1]
        %uncomment if require: \path (0,558); %set diagram left start at 0, and has height of 558

        %Shape: Ellipse [id:dp3851473480685037] 
        \draw   (118.79,76.79) .. controls (118.79,40.73) and (149.51,11.51) .. (187.4,11.51) .. controls (225.29,11.51) and (256.01,40.73) .. (256.01,76.79) .. controls (256.01,112.84) and (225.29,142.07) .. (187.4,142.07) .. controls (149.51,142.07) and (118.79,112.84) .. (118.79,76.79) -- cycle ;
        %Shape: Arc [id:dp26747108016752974] 
        \draw  [draw opacity=0] (249.84,103.26) .. controls (252.37,104.98) and (254.82,106.96) .. (257.14,109.2) .. controls (272.26,123.81) and (275.86,143.77) .. (265.19,153.76) .. controls (254.52,163.76) and (233.62,160.01) .. (218.5,145.39) .. controls (216.18,143.15) and (214.14,140.79) .. (212.37,138.36) -- (237.82,127.29) -- cycle ; \draw   (249.84,103.26) .. controls (252.37,104.98) and (254.82,106.96) .. (257.14,109.2) .. controls (272.26,123.81) and (275.86,143.77) .. (265.19,153.76) .. controls (254.52,163.76) and (233.62,160.01) .. (218.5,145.39) .. controls (216.18,143.15) and (214.14,140.79) .. (212.37,138.36) ;
        %Shape: Ellipse [id:dp048293129611794106] 
        \draw   (337.02,71.28) .. controls (337.02,35.23) and (367.73,6) .. (405.63,6) .. controls (443.52,6) and (474.24,35.23) .. (474.24,71.28) .. controls (474.24,107.34) and (443.52,136.56) .. (405.63,136.56) .. controls (367.73,136.56) and (337.02,107.34) .. (337.02,71.28) -- cycle ;
        %Shape: Arc [id:dp7733961558898084] 
        \draw  [draw opacity=0] (468.07,97.75) .. controls (470.6,99.47) and (473.05,101.45) .. (475.36,103.69) .. controls (490.48,118.31) and (494.09,138.26) .. (483.42,148.25) .. controls (472.75,158.25) and (451.85,154.5) .. (436.73,139.89) .. controls (434.41,137.65) and (432.37,135.29) .. (430.6,132.85) -- (456.05,121.79) -- cycle ; \draw   (468.07,97.75) .. controls (470.6,99.47) and (473.05,101.45) .. (475.36,103.69) .. controls (490.48,118.31) and (494.09,138.26) .. (483.42,148.25) .. controls (472.75,158.25) and (451.85,154.5) .. (436.73,139.89) .. controls (434.41,137.65) and (432.37,135.29) .. (430.6,132.85) ;
        %Straight Lines [id:da6645381549538105] 
        \draw    (273.37,67.94) -- (329.23,67.94) ;
        \draw [shift={(331.23,67.94)}, rotate = 180] [color={rgb, 255:red, 0; green, 0; blue, 0 }  ][line width=0.75]    (10.93,-3.29) .. controls (6.95,-1.4) and (3.31,-0.3) .. (0,0) .. controls (3.31,0.3) and (6.95,1.4) .. (10.93,3.29)   ;
        %Shape: Ellipse [id:dp28346841808278733] 
        \draw   (118.79,265.36) .. controls (118.79,229.3) and (149.51,200.07) .. (187.4,200.07) .. controls (225.29,200.07) and (256.01,229.3) .. (256.01,265.36) .. controls (256.01,301.41) and (225.29,330.64) .. (187.4,330.64) .. controls (149.51,330.64) and (118.79,301.41) .. (118.79,265.36) -- cycle ;
        %Shape: Arc [id:dp5467802088144775] 
        \draw  [draw opacity=0] (249.84,291.83) .. controls (252.37,293.55) and (254.82,295.53) .. (257.14,297.77) .. controls (272.26,312.38) and (275.86,332.33) .. (265.19,342.33) .. controls (254.52,352.32) and (233.62,348.58) .. (218.5,333.96) .. controls (216.18,331.72) and (214.14,329.36) .. (212.37,326.93) -- (237.82,315.86) -- cycle ; \draw   (249.84,291.83) .. controls (252.37,293.55) and (254.82,295.53) .. (257.14,297.77) .. controls (272.26,312.38) and (275.86,332.33) .. (265.19,342.33) .. controls (254.52,352.32) and (233.62,348.58) .. (218.5,333.96) .. controls (216.18,331.72) and (214.14,329.36) .. (212.37,326.93) ;
        %Straight Lines [id:da2550847089224184] 
        \draw    (331.95,183.56) -- (264.05,225.18) ;
        \draw [shift={(262.35,226.23)}, rotate = 328.49] [color={rgb, 255:red, 0; green, 0; blue, 0 }  ][line width=0.75]    (10.93,-3.29) .. controls (6.95,-1.4) and (3.31,-0.3) .. (0,0) .. controls (3.31,0.3) and (6.95,1.4) .. (10.93,3.29)   ;
        %Shape: Ellipse [id:dp8651156930485036] 
        \draw   (342.82,251.59) .. controls (342.82,215.54) and (373.53,186.31) .. (411.43,186.31) .. controls (449.32,186.31) and (480.04,215.54) .. (480.04,251.59) .. controls (480.04,287.65) and (449.32,316.87) .. (411.43,316.87) .. controls (373.53,316.87) and (342.82,287.65) .. (342.82,251.59) -- cycle ;
        %Shape: Arc [id:dp4575747283593845] 
        \draw  [draw opacity=0] (473.87,278.06) .. controls (476.4,279.78) and (478.85,281.76) .. (481.16,284) .. controls (496.28,298.62) and (499.89,318.57) .. (489.22,328.57) .. controls (478.55,338.56) and (457.65,334.81) .. (442.53,320.2) .. controls (440.21,317.96) and (438.17,315.6) .. (436.4,313.16) -- (461.85,302.1) -- cycle ; \draw   (473.87,278.06) .. controls (476.4,279.78) and (478.85,281.76) .. (481.16,284) .. controls (496.28,298.62) and (499.89,318.57) .. (489.22,328.57) .. controls (478.55,338.56) and (457.65,334.81) .. (442.53,320.2) .. controls (440.21,317.96) and (438.17,315.6) .. (436.4,313.16) ;
        %Straight Lines [id:da9619731386949268] 
        \draw    (271.92,268.89) -- (330.68,269.03) ;
        \draw [shift={(332.68,269.03)}, rotate = 180.13] [color={rgb, 255:red, 0; green, 0; blue, 0 }  ][line width=0.75]    (10.93,-3.29) .. controls (6.95,-1.4) and (3.31,-0.3) .. (0,0) .. controls (3.31,0.3) and (6.95,1.4) .. (10.93,3.29)   ;
        %Shape: Ellipse [id:dp5569398050814263] 
        \draw   (123.14,438.78) .. controls (123.14,402.73) and (153.86,373.5) .. (191.75,373.5) .. controls (229.64,373.5) and (260.36,402.73) .. (260.36,438.78) .. controls (260.36,474.84) and (229.64,504.07) .. (191.75,504.07) .. controls (153.86,504.07) and (123.14,474.84) .. (123.14,438.78) -- cycle ;
        %Shape: Arc [id:dp7522127717600136] 
        \draw  [draw opacity=0] (254.19,465.26) .. controls (256.72,466.97) and (259.17,468.96) .. (261.49,471.19) .. controls (276.61,485.81) and (280.21,505.76) .. (269.54,515.76) .. controls (258.87,525.75) and (237.97,522.01) .. (222.85,507.39) .. controls (220.53,505.15) and (218.49,502.79) .. (216.72,500.36) -- (242.17,489.29) -- cycle ; \draw   (254.19,465.26) .. controls (256.72,466.97) and (259.17,468.96) .. (261.49,471.19) .. controls (276.61,485.81) and (280.21,505.76) .. (269.54,515.76) .. controls (258.87,525.75) and (237.97,522.01) .. (222.85,507.39) .. controls (220.53,505.15) and (218.49,502.79) .. (216.72,500.36) ;
        %Straight Lines [id:da36554158474057674] 
        \draw    (350.08,371.58) -- (280.01,413.76) ;
        \draw [shift={(278.3,414.79)}, rotate = 328.95] [color={rgb, 255:red, 0; green, 0; blue, 0 }  ][line width=0.75]    (10.93,-3.29) .. controls (6.95,-1.4) and (3.31,-0.3) .. (0,0) .. controls (3.31,0.3) and (6.95,1.4) .. (10.93,3.29)   ;

        % Text Node
        \draw (222.43,107.52) node [anchor=north west][inner sep=0.75pt]   [align=left] {$\displaystyle 0$};
        % Text Node
        \draw (192.79,122.71) node [anchor=north west][inner sep=0.75pt]   [align=left] {$\displaystyle 1$};
        % Text Node
        \draw (170.95,121.6) node [anchor=north west][inner sep=0.75pt]   [align=left] {$\displaystyle 2$};
        % Text Node
        \draw (122.45,76.99) node [anchor=north west][inner sep=0.75pt]   [align=left] {...};
        % Text Node
        \draw (126.18,40.82) node [anchor=north west][inner sep=0.75pt]   [align=left] {$\displaystyle j-1$};
        % Text Node
        \draw (157.1,21.81) node [anchor=north west][inner sep=0.75pt]   [align=left] {$\displaystyle j$};
        % Text Node
        \draw (105.36,16.61) node [anchor=north west][inner sep=0.75pt]   [align=left] {...};
        % Text Node
        \draw (217.37,55.09) node [anchor=north west][inner sep=0.75pt]   [align=left] {$\displaystyle k-1$};
        % Text Node
        \draw (239.19,82.75) node [anchor=north west][inner sep=0.75pt]   [align=left] {$\displaystyle k$};
        % Text Node
        \draw (181.72,146.88) node [anchor=north west][inner sep=0.75pt]   [align=left] {$\displaystyle k+1$};
        % Text Node
        \draw (211.07,159.05) node [anchor=north west][inner sep=0.75pt]   [align=left] {$\displaystyle k+2$};
        % Text Node
        \draw (272.02,114.37) node [anchor=north west][inner sep=0.75pt]   [align=left] {$\displaystyle k+\ell -1$};
        % Text Node
        \draw (265.85,153.61) node [anchor=north west][inner sep=0.75pt]   [align=left] {...};
        % Text Node
        \draw (259.24,92.02) node [anchor=north west][inner sep=0.75pt]   [align=left] {$\displaystyle k+\ell $};
        % Text Node
        \draw (402.13,139.38) node [anchor=north west][inner sep=0.75pt]   [align=left] {$\displaystyle k+1$};
        % Text Node
        \draw (430.85,158.24) node [anchor=north west][inner sep=0.75pt]   [align=left] {$\displaystyle k+2$};
        % Text Node
        \draw (490.52,110.17) node [anchor=north west][inner sep=0.75pt]   [align=left] {$\displaystyle k+\ell -1$};
        % Text Node
        \draw (489.16,150.17) node [anchor=north west][inner sep=0.75pt]   [align=left] {...};
        % Text Node
        \draw (481.29,81.21) node [anchor=north west][inner sep=0.75pt]   [align=left] {$\displaystyle k+\ell $};
        % Text Node
        \draw (285.2,43.55) node [anchor=north west][inner sep=0.75pt]   [align=left] {$\displaystyle \sigma _{1}^{k-j}$};
        % Text Node
        \draw (440.56,102.95) node [anchor=north west][inner sep=0.75pt]   [align=left] {$\displaystyle j$};
        % Text Node
        \draw (433.56,67.03) node [anchor=north west][inner sep=0.75pt]   [align=left] {$\displaystyle j-1$};
        % Text Node
        \draw (430.94,41.63) node [anchor=north west][inner sep=0.75pt]   [align=left] {$\displaystyle j-2$};
        % Text Node
        \draw (389.44,114.03) node [anchor=north west][inner sep=0.75pt]   [align=left] {$\displaystyle j+1$};
        % Text Node
        \draw (358.27,92.39) node [anchor=north west][inner sep=0.75pt]   [align=left] {$\displaystyle j+2$};
        % Text Node
        \draw (366.05,26.76) node [anchor=north west][inner sep=0.75pt]   [align=left] {...};
        % Text Node
        \draw (203.5,293.91) node [anchor=north west][inner sep=0.75pt]   [align=left] {$\displaystyle k+1$};
        % Text Node
        \draw (177.34,333.48) node [anchor=north west][inner sep=0.75pt]   [align=left] {$\displaystyle k+2$};
        % Text Node
        \draw (275.02,323.95) node [anchor=north west][inner sep=0.75pt]   [align=left] {$\displaystyle k+\ell -1$};
        % Text Node
        \draw (265.13,344.93) node [anchor=north west][inner sep=0.75pt]   [align=left] {...};
        % Text Node
        \draw (270.39,298.67) node [anchor=north west][inner sep=0.75pt]   [align=left] {$\displaystyle k+\ell $};
        % Text Node
        \draw (266.3,282.26) node [anchor=north west][inner sep=0.75pt]   [align=left] {$\displaystyle j$};
        % Text Node
        \draw (220.43,265.86) node [anchor=north west][inner sep=0.75pt]   [align=left] {$\displaystyle j-1$};
        % Text Node
        \draw (218.16,238.46) node [anchor=north west][inner sep=0.75pt]   [align=left] {$\displaystyle j-2$};
        % Text Node
        \draw (151.61,299.04) node [anchor=north west][inner sep=0.75pt]   [align=left] {$\displaystyle j+1$};
        % Text Node
        \draw (128.41,277.27) node [anchor=north west][inner sep=0.75pt]   [align=left] {$\displaystyle j+2$};
        % Text Node
        \draw (147.82,220.83) node [anchor=north west][inner sep=0.75pt]   [align=left] {...};
        % Text Node
        \draw (284.53,179.91) node [anchor=north west][inner sep=0.75pt]   [align=left] {$\displaystyle \sigma _{2}$};
        % Text Node
        \draw (207.27,349.87) node [anchor=north west][inner sep=0.75pt]   [align=left] {$\displaystyle k+3$};
        % Text Node
        \draw (364.08,275.34) node [anchor=north west][inner sep=0.75pt]   [align=left] {$\displaystyle k+1$};
        % Text Node
        \draw (408.1,319.47) node [anchor=north west][inner sep=0.75pt]   [align=left] {$\displaystyle k+2$};
        % Text Node
        \draw (499.6,308.94) node [anchor=north west][inner sep=0.75pt]   [align=left] {$\displaystyle k+\ell -1$};
        % Text Node
        \draw (489.88,329.79) node [anchor=north west][inner sep=0.75pt]   [align=left] {...};
        % Text Node
        \draw (495.24,285.9) node [anchor=north west][inner sep=0.75pt]   [align=left] {$\displaystyle k+\ell $};
        % Text Node
        \draw (491.88,264.5) node [anchor=north west][inner sep=0.75pt]   [align=left] {$\displaystyle j$};
        % Text Node
        \draw (387.61,291.83) node [anchor=north west][inner sep=0.75pt]   [align=left] {$\displaystyle j-1$};
        % Text Node
        \draw (427.69,279) node [anchor=north west][inner sep=0.75pt]   [align=left] {$\displaystyle j-2$};
        % Text Node
        \draw (351.29,249.75) node [anchor=north west][inner sep=0.75pt]   [align=left] {$\displaystyle j+1$};
        % Text Node
        \draw (348.77,225.23) node [anchor=north west][inner sep=0.75pt]   [align=left] {$\displaystyle j+2$};
        % Text Node
        \draw (371.85,207.07) node [anchor=north west][inner sep=0.75pt]   [align=left] {...};
        % Text Node
        \draw (444.37,334.86) node [anchor=north west][inner sep=0.75pt]   [align=left] {$\displaystyle k+3$};
        % Text Node
        \draw (296.86,247.57) node [anchor=north west][inner sep=0.75pt]   [align=left] {$\displaystyle \sigma _{1}^{2}$};
        % Text Node
        \draw (444.81,250.03) node [anchor=north west][inner sep=0.75pt]   [align=left] {$\displaystyle j-3$};
        % Text Node
        \draw (438.19,224.38) node [anchor=north west][inner sep=0.75pt]   [align=left] {$\displaystyle j-4$};
        % Text Node
        \draw (143.95,466.84) node [anchor=north west][inner sep=0.75pt]   [align=left] {$\displaystyle k+1$};
        % Text Node
        \draw (222.07,524.99) node [anchor=north west][inner sep=0.75pt]   [align=left] {$\displaystyle k+2$};
        % Text Node
        \draw (279.82,480.05) node [anchor=north west][inner sep=0.75pt]   [align=left] {$\displaystyle k+\ell -1$};
        % Text Node
        \draw (270.93,514.91) node [anchor=north west][inner sep=0.75pt]   [align=left] {...};
        % Text Node
        \draw (268.76,455.46) node [anchor=north west][inner sep=0.75pt]   [align=left] {$\displaystyle k+\ell $};
        % Text Node
        \draw (223.78,469.77) node [anchor=north west][inner sep=0.75pt]   [align=left] {$\displaystyle j$};
        % Text Node
        \draw (172.66,482.33) node [anchor=north west][inner sep=0.75pt]   [align=left] {$\displaystyle j-1$};
        % Text Node
        \draw (189.75,511.07) node [anchor=north west][inner sep=0.75pt]   [align=left] {$\displaystyle j-2$};
        % Text Node
        \draw (131.44,440.63) node [anchor=north west][inner sep=0.75pt]   [align=left] {$\displaystyle j+1$};
        % Text Node
        \draw (129.26,418.07) node [anchor=north west][inner sep=0.75pt]   [align=left] {$\displaystyle j+2$};
        % Text Node
        \draw (157.97,392.88) node [anchor=north west][inner sep=0.75pt]   [align=left] {...};
        % Text Node
        \draw (223.23,437.79) node [anchor=north west][inner sep=0.75pt]   [align=left] {$\displaystyle j-3$};
        % Text Node
        \draw (222.96,413.64) node [anchor=north west][inner sep=0.75pt]   [align=left] {$\displaystyle j-4$};
        % Text Node
        \draw (302.42,363.88) node [anchor=north west][inner sep=0.75pt]   [align=left] {$\displaystyle \sigma _{2}^{-1}$};
        % Text Node
        % \draw (370,391) node [anchor=north west][inner sep=0.75pt]   [align=left] {Note that everything is in\\order from $\displaystyle j+1$ to $\displaystyle j-3$\\and $\displaystyle k+1$ is always in the\\second position.};
        % Text Node
        \draw (312,415) node [anchor=north west][inner sep=0.75pt]   [align=left] {($\displaystyle \star $)};
        % Text Node
        \draw (55,26.4) node [anchor=north west][inner sep=0.75pt]    {};

    \end{tikzpicture}
    
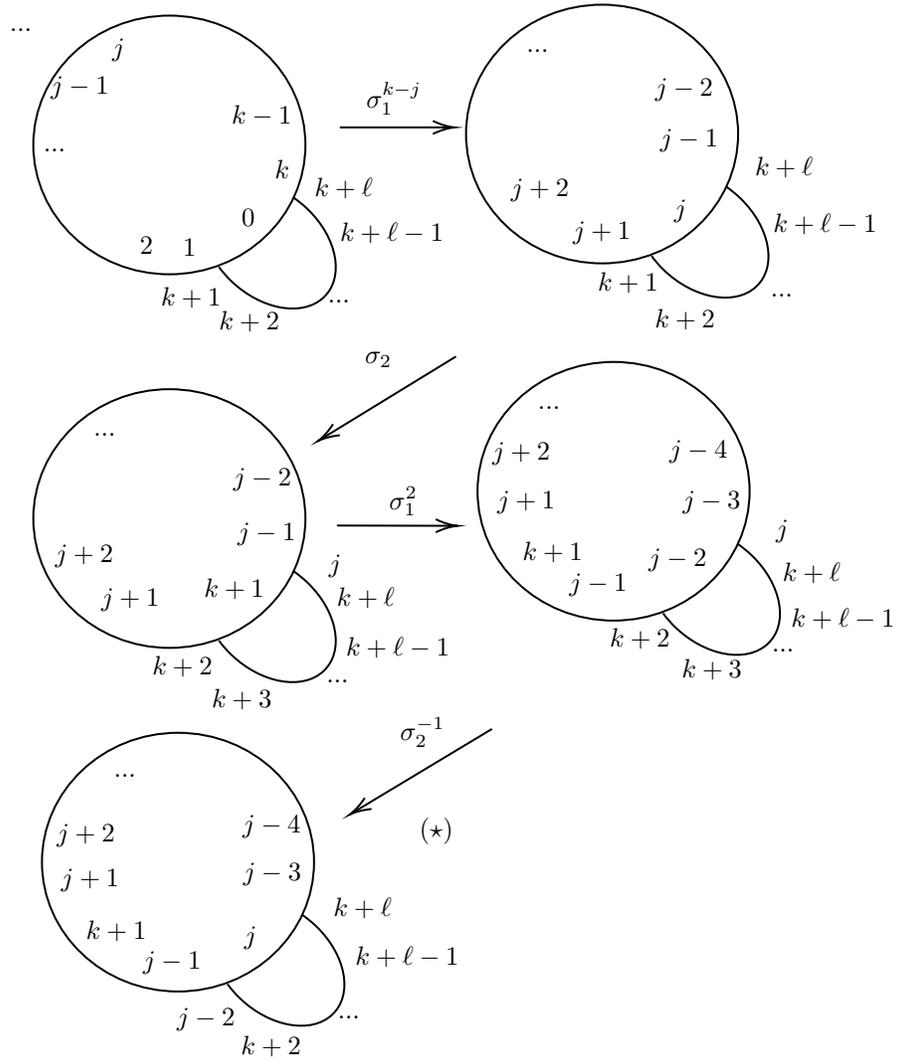
\captionof{figure}{Rubik's Shape after $\sigma_2^{-1}\sigma_1^2\sigma_2\sigma_1^{k+1-j}$}

    In between the edges labeled in the diagram, the unlabeled edges are in the natural order and the $k+1$ edge is always in position $2$.

    \begin{tikzpicture}[x=0.75pt,y=0.75pt,yscale=-1,xscale=1]
        %uncomment if require: \path (0,1167); %set diagram left start at 0, and has height of 1167

        %Shape: Ellipse [id:dp11492936941010012] 
        \draw   (119.79,72.74) .. controls (119.79,36.41) and (149.5,6.95) .. (186.14,6.95) .. controls (222.79,6.95) and (252.5,36.41) .. (252.5,72.74) .. controls (252.5,109.07) and (222.79,138.52) .. (186.14,138.52) .. controls (149.5,138.52) and (119.79,109.07) .. (119.79,72.74) -- cycle ;
        %Shape: Arc [id:dp6184922186147035] 
        \draw  [draw opacity=0] (246.53,99.41) .. controls (248.98,101.14) and (251.35,103.14) .. (253.59,105.39) .. controls (268.22,120.12) and (271.71,140.22) .. (261.39,150.3) .. controls (251.07,160.37) and (230.84,156.59) .. (216.22,141.87) .. controls (213.98,139.61) and (212,137.23) .. (210.29,134.77) -- (234.91,123.63) -- cycle ; \draw   (246.53,99.41) .. controls (248.98,101.14) and (251.35,103.14) .. (253.59,105.39) .. controls (268.22,120.12) and (271.71,140.22) .. (261.39,150.3) .. controls (251.07,160.37) and (230.84,156.59) .. (216.22,141.87) .. controls (213.98,139.61) and (212,137.23) .. (210.29,134.77) ;
        %Straight Lines [id:da0844113158767279] 
        \draw    (274.9,61.08) -- (326.75,60.41) ;
        \draw [shift={(328.75,60.38)}, rotate = 179.26] [color={rgb, 255:red, 0; green, 0; blue, 0 }  ][line width=0.75]    (10.93,-3.29) .. controls (6.95,-1.4) and (3.31,-0.3) .. (0,0) .. controls (3.31,0.3) and (6.95,1.4) .. (10.93,3.29)   ;
        %Shape: Ellipse [id:dp12472713948684278] 
        \draw   (332.95,72.74) .. controls (332.95,36.41) and (362.66,6.95) .. (399.31,6.95) .. controls (435.96,6.95) and (465.66,36.41) .. (465.66,72.74) .. controls (465.66,109.07) and (435.96,138.52) .. (399.31,138.52) .. controls (362.66,138.52) and (332.95,109.07) .. (332.95,72.74) -- cycle ;
        %Shape: Arc [id:dp15059287803772037] 
        \draw  [draw opacity=0] (459.7,99.41) .. controls (462.15,101.14) and (464.52,103.14) .. (466.76,105.39) .. controls (481.38,120.12) and (484.87,140.22) .. (474.55,150.3) .. controls (464.23,160.37) and (444.01,156.59) .. (429.38,141.87) .. controls (427.14,139.61) and (425.16,137.23) .. (423.45,134.77) -- (448.07,123.63) -- cycle ; \draw   (459.7,99.41) .. controls (462.15,101.14) and (464.52,103.14) .. (466.76,105.39) .. controls (481.38,120.12) and (484.87,140.22) .. (474.55,150.3) .. controls (464.23,160.37) and (444.01,156.59) .. (429.38,141.87) .. controls (427.14,139.61) and (425.16,137.23) .. (423.45,134.77) ;
        %Shape: Ellipse [id:dp18020224189504397] 
        \draw   (126.63,248.88) .. controls (126.63,212.55) and (156.34,183.1) .. (192.99,183.1) .. controls (229.63,183.1) and (259.34,212.55) .. (259.34,248.88) .. controls (259.34,285.21) and (229.63,314.66) .. (192.99,314.66) .. controls (156.34,314.66) and (126.63,285.21) .. (126.63,248.88) -- cycle ;
        %Shape: Arc [id:dp24785801004185548] 
        \draw  [draw opacity=0] (253.38,275.55) .. controls (255.82,277.28) and (258.2,279.28) .. (260.44,281.53) .. controls (275.06,296.26) and (278.55,316.37) .. (268.23,326.44) .. controls (257.91,336.51) and (237.69,332.74) .. (223.06,318.01) .. controls (220.82,315.75) and (218.84,313.37) .. (217.13,310.92) -- (241.75,299.77) -- cycle ; \draw   (253.38,275.55) .. controls (255.82,277.28) and (258.2,279.28) .. (260.44,281.53) .. controls (275.06,296.26) and (278.55,316.37) .. (268.23,326.44) .. controls (257.91,336.51) and (237.69,332.74) .. (223.06,318.01) .. controls (220.82,315.75) and (218.84,313.37) .. (217.13,310.92) ;
        %Straight Lines [id:da7880004491491575] 
        \draw    (330.16,181.74) -- (270.21,213.81) ;
        \draw [shift={(268.45,214.75)}, rotate = 331.86] [color={rgb, 255:red, 0; green, 0; blue, 0 }  ][line width=0.75]    (10.93,-3.29) .. controls (6.95,-1.4) and (3.31,-0.3) .. (0,0) .. controls (3.31,0.3) and (6.95,1.4) .. (10.93,3.29)   ;
        %Shape: Ellipse [id:dp9706999219367343] 
        \draw   (335.59,248.18) .. controls (335.59,211.85) and (365.3,182.4) .. (401.94,182.4) .. controls (438.59,182.4) and (468.3,211.85) .. (468.3,248.18) .. controls (468.3,284.51) and (438.59,313.97) .. (401.94,313.97) .. controls (365.3,313.97) and (335.59,284.51) .. (335.59,248.18) -- cycle ;
        %Shape: Arc [id:dp00727643701865599] 
        \draw  [draw opacity=0] (462.33,274.85) .. controls (464.78,276.59) and (467.15,278.58) .. (469.39,280.84) .. controls (484.02,295.57) and (487.51,315.67) .. (477.19,325.74) .. controls (466.87,335.82) and (446.64,332.04) .. (432.02,317.31) .. controls (429.78,315.06) and (427.8,312.68) .. (426.09,310.22) -- (450.71,299.08) -- cycle ; \draw   (462.33,274.85) .. controls (464.78,276.59) and (467.15,278.58) .. (469.39,280.84) .. controls (484.02,295.57) and (487.51,315.67) .. (477.19,325.74) .. controls (466.87,335.82) and (446.64,332.04) .. (432.02,317.31) .. controls (429.78,315.06) and (427.8,312.68) .. (426.09,310.22) ;
        %Straight Lines [id:da8883833986301954] 
        \draw    (272.8,253.58) -- (329.56,253.85) ;
        \draw [shift={(331.56,253.86)}, rotate = 180.27] [color={rgb, 255:red, 0; green, 0; blue, 0 }  ][line width=0.75]    (10.93,-3.29) .. controls (6.95,-1.4) and (3.31,-0.3) .. (0,0) .. controls (3.31,0.3) and (6.95,1.4) .. (10.93,3.29)   ;
        %Shape: Ellipse [id:dp5228009768038788] 
        \draw   (128.33,420.86) .. controls (128.33,384.53) and (158.04,355.08) .. (194.69,355.08) .. controls (231.33,355.08) and (261.04,384.53) .. (261.04,420.86) .. controls (261.04,457.19) and (231.33,486.64) .. (194.69,486.64) .. controls (158.04,486.64) and (128.33,457.19) .. (128.33,420.86) -- cycle ;
        %Shape: Arc [id:dp17179557720291494] 
        \draw  [draw opacity=0] (255.08,447.53) .. controls (257.52,449.26) and (259.9,451.26) .. (262.14,453.52) .. controls (276.76,468.24) and (280.25,488.35) .. (269.93,498.42) .. controls (259.61,508.49) and (239.39,504.72) .. (224.76,489.99) .. controls (222.52,487.73) and (220.54,485.35) .. (218.83,482.9) -- (243.45,471.75) -- cycle ; \draw   (255.08,447.53) .. controls (257.52,449.26) and (259.9,451.26) .. (262.14,453.52) .. controls (276.76,468.24) and (280.25,488.35) .. (269.93,498.42) .. controls (259.61,508.49) and (239.39,504.72) .. (224.76,489.99) .. controls (222.52,487.73) and (220.54,485.35) .. (218.83,482.9) ;
        %Straight Lines [id:da8266227533489909] 
        \draw    (342.78,362.04) -- (281.43,395.48) ;
        \draw [shift={(279.67,396.44)}, rotate = 331.41] [color={rgb, 255:red, 0; green, 0; blue, 0 }  ][line width=0.75]    (10.93,-3.29) .. controls (6.95,-1.4) and (3.31,-0.3) .. (0,0) .. controls (3.31,0.3) and (6.95,1.4) .. (10.93,3.29)   ;

        % Text Node
        \draw (138.4,95.74) node [anchor=north west][inner sep=0.75pt]   [align=left] {$\displaystyle k+1$};
        % Text Node
        \draw (209.13,156.7) node [anchor=north west][inner sep=0.75pt]   [align=left] {$\displaystyle k+2$};
        % Text Node
        \draw (267.2,114.7) node [anchor=north west][inner sep=0.75pt]   [align=left] {$\displaystyle k+\ell -1$};
        % Text Node
        \draw (260.39,150.21) node [anchor=north west][inner sep=0.75pt]   [align=left] {...};
        % Text Node
        \draw (255.27,86.27) node [anchor=north west][inner sep=0.75pt]   [align=left] {$\displaystyle k+\ell $};
        % Text Node
        \draw (216.94,101.95) node [anchor=north west][inner sep=0.75pt]   [align=left] {$\displaystyle j$};
        % Text Node
        \draw (163.96,113.7) node [anchor=north west][inner sep=0.75pt]   [align=left] {$\displaystyle j-1$};
        % Text Node
        \draw (181.08,141.58) node [anchor=north west][inner sep=0.75pt]   [align=left] {$\displaystyle j-2$};
        % Text Node
        \draw (128.56,75.69) node [anchor=north west][inner sep=0.75pt]   [align=left] {$\displaystyle j+1$};
        % Text Node
        \draw (123.05,50.89) node [anchor=north west][inner sep=0.75pt]   [align=left] {$\displaystyle j+2$};
        % Text Node
        \draw (153.25,26.55) node [anchor=north west][inner sep=0.75pt]   [align=left] {...};
        % Text Node
        \draw (214.94,71.17) node [anchor=north west][inner sep=0.75pt]   [align=left] {$\displaystyle j-3$};
        % Text Node
        \draw (214.64,45.59) node [anchor=north west][inner sep=0.75pt]   [align=left] {$\displaystyle j-4$};
        % Text Node
        \draw (340.05,83.26) node [anchor=north west][inner sep=0.75pt]   [align=left] {$\displaystyle k+1$};
        % Text Node
        \draw (428.49,159.39) node [anchor=north west][inner sep=0.75pt]   [align=left] {$\displaystyle k+2$};
        % Text Node
        \draw (483.36,119.41) node [anchor=north west][inner sep=0.75pt]   [align=left] {$\displaystyle k+\ell -1$};
        % Text Node
        \draw (473.55,150.9) node [anchor=north west][inner sep=0.75pt]   [align=left] {...};
        % Text Node
        \draw (476.33,87.66) node [anchor=north west][inner sep=0.75pt]   [align=left] {$\displaystyle k+\ell $};
        % Text Node
        \draw (392.15,114.28) node [anchor=north west][inner sep=0.75pt]   [align=left] {$\displaystyle j$};
        % Text Node
        \draw (357.89,103.76) node [anchor=north west][inner sep=0.75pt]   [align=left] {$\displaystyle j-1$};
        % Text Node
        \draw (395.31,145.52) node [anchor=north west][inner sep=0.75pt]   [align=left] {$\displaystyle j-2$};
        % Text Node
        \draw (338.75,55.38) node [anchor=north west][inner sep=0.75pt]   [align=left] {$\displaystyle j+1$};
        % Text Node
        \draw (366.41,26.55) node [anchor=north west][inner sep=0.75pt]   [align=left] {...};
        % Text Node
        \draw (416.88,96.13) node [anchor=north west][inner sep=0.75pt]   [align=left] {$\displaystyle j-3$};
        % Text Node
        \draw (429.2,68.7) node [anchor=north west][inner sep=0.75pt]   [align=left] {$\displaystyle j-4$};
        % Text Node
        \draw (293.8,41.92) node [anchor=north west][inner sep=0.75pt]    {$\sigma _{1}$};
        % Text Node
        \draw (426.4,45.51) node [anchor=north west][inner sep=0.75pt]   [align=left] {$\displaystyle j-5$};
        % Text Node
        \draw (137.73,259.4) node [anchor=north west][inner sep=0.75pt]   [align=left] {$\displaystyle k+1$};
        % Text Node
        \draw (191.03,321.36) node [anchor=north west][inner sep=0.75pt]   [align=left] {$\displaystyle k+2$};
        % Text Node
        \draw (268.63,326.35) node [anchor=north west][inner sep=0.75pt]   [align=left] {...};
        % Text Node
        \draw (281.03,292.6) node [anchor=north west][inner sep=0.75pt]   [align=left] {$\displaystyle k+\ell $};
        % Text Node
        \draw (190.83,295.43) node [anchor=north west][inner sep=0.75pt]   [align=left] {$\displaystyle j$};
        % Text Node
        \draw (157.57,280.9) node [anchor=north west][inner sep=0.75pt]   [align=left] {$\displaystyle j-1$};
        % Text Node
        \draw (209.96,276.5) node [anchor=north west][inner sep=0.75pt]   [align=left] {$\displaystyle j-2$};
        % Text Node
        \draw (132.79,237.19) node [anchor=north west][inner sep=0.75pt]   [align=left] {$\displaystyle j+1$};
        % Text Node
        \draw (160.09,202.69) node [anchor=north west][inner sep=0.75pt]   [align=left] {...};
        % Text Node
        \draw (266.54,262.26) node [anchor=north west][inner sep=0.75pt]   [align=left] {$\displaystyle j-3$};
        % Text Node
        \draw (221.88,248.07) node [anchor=north west][inner sep=0.75pt]   [align=left] {$\displaystyle j-4$};
        % Text Node
        \draw (222.58,226.65) node [anchor=north west][inner sep=0.75pt]   [align=left] {$\displaystyle j-5$};
        % Text Node
        \draw (293.1,171.6) node [anchor=north west][inner sep=0.75pt]    {$\sigma _{2}$};
        % Text Node
        \draw (221.67,335.14) node [anchor=north west][inner sep=0.75pt]   [align=left] {$\displaystyle k+3$};
        % Text Node
        \draw (343.67,222.13) node [anchor=north west][inner sep=0.75pt]   [align=left] {$\displaystyle k+1$};
        % Text Node
        \draw (397.09,320.66) node [anchor=north west][inner sep=0.75pt]   [align=left] {$\displaystyle k+2$};
        % Text Node
        \draw (476.89,324.96) node [anchor=north west][inner sep=0.75pt]   [align=left] {...};
        % Text Node
        \draw (487.69,291.91) node [anchor=north west][inner sep=0.75pt]   [align=left] {$\displaystyle k+\ell $};
        % Text Node
        \draw (362.73,272.41) node [anchor=north west][inner sep=0.75pt]   [align=left] {$\displaystyle j$};
        % Text Node
        \draw (344.19,249.79) node [anchor=north west][inner sep=0.75pt]   [align=left] {$\displaystyle j-1$};
        % Text Node
        \draw (377.46,290.15) node [anchor=north west][inner sep=0.75pt]   [align=left] {$\displaystyle j-2$};
        % Text Node
        \draw (369.05,202) node [anchor=north west][inner sep=0.75pt]   [align=left] {...};
        % Text Node
        \draw (477.5,265.57) node [anchor=north west][inner sep=0.75pt]   [align=left] {$\displaystyle j-3$};
        % Text Node
        \draw (417.82,273.5) node [anchor=north west][inner sep=0.75pt]   [align=left] {$\displaystyle j-4$};
        % Text Node
        \draw (431.84,243.22) node [anchor=north west][inner sep=0.75pt]   [align=left] {$\displaystyle j-5$};
        % Text Node
        \draw (432.03,334.45) node [anchor=north west][inner sep=0.75pt]   [align=left] {$\displaystyle k+3$};
        % Text Node
        \draw (298.71,233.04) node [anchor=north west][inner sep=0.75pt]    {$\sigma _{1}$};
        % Text Node
        \draw (426.14,216.49) node [anchor=north west][inner sep=0.75pt]   [align=left] {$\displaystyle j-6$};
        % Text Node
        \draw (135.42,408.8) node [anchor=north west][inner sep=0.75pt]   [align=left] {$\displaystyle k+1$};
        % Text Node
        \draw (225.97,509.13) node [anchor=north west][inner sep=0.75pt]   [align=left] {$\displaystyle k+2$};
        % Text Node
        \draw (271.03,499.02) node [anchor=north west][inner sep=0.75pt]   [align=left] {...};
        % Text Node
        \draw (266.6,439.85) node [anchor=north west][inner sep=0.75pt]   [align=left] {$\displaystyle k+\ell $};
        % Text Node
        \draw (157.48,450.39) node [anchor=north west][inner sep=0.75pt]   [align=left] {$ $j};
        % Text Node
        \draw (140.93,432.47) node [anchor=north west][inner sep=0.75pt]   [align=left] {$\displaystyle j-1$};
        % Text Node
        \draw (174.2,464.43) node [anchor=north west][inner sep=0.75pt]   [align=left] {$\displaystyle j-2$};
        % Text Node
        \draw (161.79,374.67) node [anchor=north west][inner sep=0.75pt]   [align=left] {...};
        % Text Node
        \draw (213.16,448.18) node [anchor=north west][inner sep=0.75pt]   [align=left] {$\displaystyle j-3$};
        % Text Node
        \draw (188.69,492.64) node [anchor=north west][inner sep=0.75pt]   [align=left] {$\displaystyle j-4$};
        % Text Node
        \draw (222.28,411.58) node [anchor=north west][inner sep=0.75pt]   [align=left] {$\displaystyle j-5$};
        % Text Node
        \draw (217.78,392.78) node [anchor=north west][inner sep=0.75pt]   [align=left] {$\displaystyle j-6$};
        % Text Node
        \draw (299.77,350.75) node [anchor=north west][inner sep=0.75pt]    {$\sigma _{2}^{-1}$};
        % Text Node
        \draw (280.04,468.82) node [anchor=north west][inner sep=0.75pt]   [align=left] {$\displaystyle k+\ell -1$};
        % Text Node
        \draw (55,26.4) node [anchor=north west][inner sep=0.75pt]    {};

    \end{tikzpicture}
    
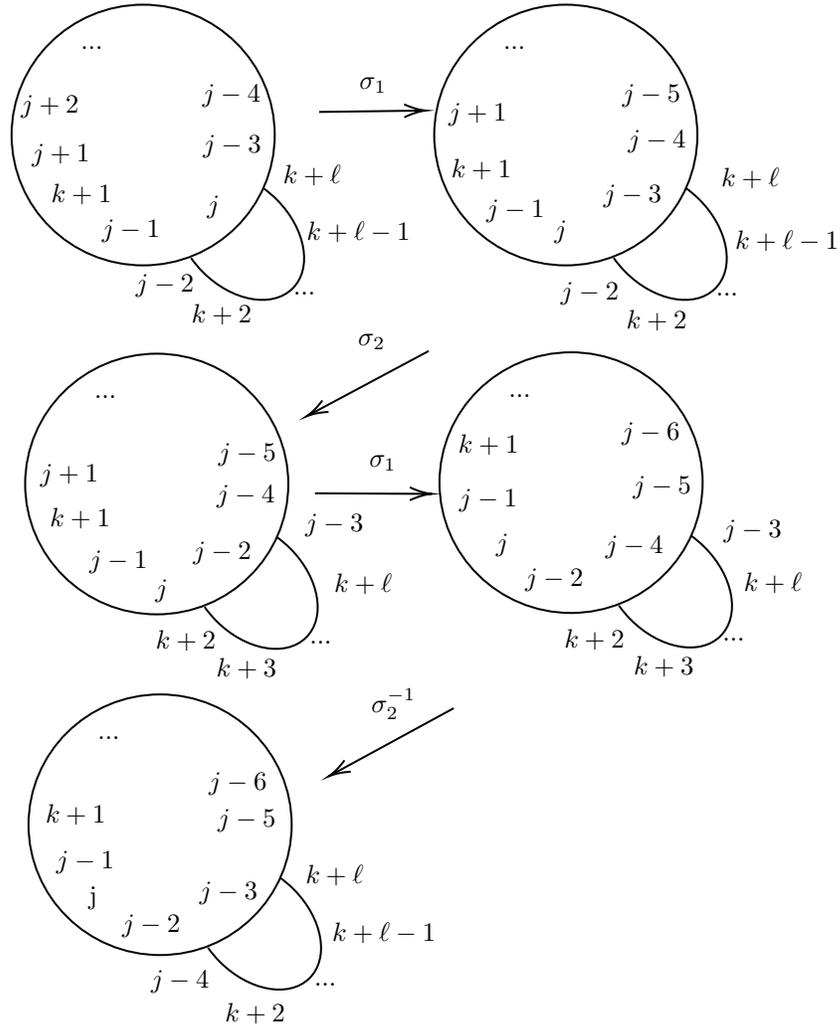
\captionof{figure}{Run $\sigma_2^{-1}\sigma_1\sigma_2\sigma_1$ on $(\star)$}

    Note that $k+1$ always moves two positions. Since $k$ is even, it takes $\frac{k-2}{2}$ to move $k+1$ to it's original spot. The $k-2$ comes from the fact that $k+1$ is always in the second position in $(\star)$.

    \begin{tikzpicture}[x=0.75pt,y=0.75pt,yscale=-1,xscale=1]
        %uncomment if require: \path (0,365); %set diagram left start at 0, and has height of 365

        %Shape: Ellipse [id:dp14319050725254123] 
        \draw   (128.79,82.32) .. controls (128.79,44.83) and (160.95,14.44) .. (200.63,14.44) .. controls (240.3,14.44) and (272.47,44.83) .. (272.47,82.32) .. controls (272.47,119.81) and (240.3,150.2) .. (200.63,150.2) .. controls (160.95,150.2) and (128.79,119.81) .. (128.79,82.32) -- cycle ;
        %Shape: Arc [id:dp9765775989728762] 
        \draw  [draw opacity=0] (266.01,109.85) .. controls (268.66,111.63) and (271.23,113.69) .. (273.65,116.02) .. controls (289.48,131.22) and (293.26,151.96) .. (282.09,162.35) .. controls (270.91,172.74) and (249.02,168.85) .. (233.19,153.65) .. controls (230.76,151.32) and (228.62,148.86) .. (226.78,146.34) -- (253.42,134.83) -- cycle ; \draw   (266.01,109.85) .. controls (268.66,111.63) and (271.23,113.69) .. (273.65,116.02) .. controls (289.48,131.22) and (293.26,151.96) .. (282.09,162.35) .. controls (270.91,172.74) and (249.02,168.85) .. (233.19,153.65) .. controls (230.76,151.32) and (228.62,148.86) .. (226.78,146.34) ;
        %Straight Lines [id:da8496874476382688] 
        \draw    (292.16,77.87) -- (350.74,78.84) ;
        \draw [shift={(352.74,78.87)}, rotate = 180.95] [color={rgb, 255:red, 0; green, 0; blue, 0 }  ][line width=0.75]    (10.93,-3.29) .. controls (6.95,-1.4) and (3.31,-0.3) .. (0,0) .. controls (3.31,0.3) and (6.95,1.4) .. (10.93,3.29)   ;
        %Shape: Ellipse [id:dp09517951855415263] 
        \draw   (364.12,82.78) .. controls (364.12,45.29) and (396.28,14.9) .. (435.96,14.9) .. controls (475.63,14.9) and (507.8,45.29) .. (507.8,82.78) .. controls (507.8,120.27) and (475.63,150.66) .. (435.96,150.66) .. controls (396.28,150.66) and (364.12,120.27) .. (364.12,82.78) -- cycle ;
        %Shape: Arc [id:dp8239879654464959] 
        \draw  [draw opacity=0] (501.35,110.31) .. controls (503.99,112.09) and (506.56,114.16) .. (508.98,116.48) .. controls (524.81,131.68) and (528.59,152.43) .. (517.42,162.82) .. controls (506.25,173.21) and (484.35,169.31) .. (468.52,154.11) .. controls (466.1,151.78) and (463.95,149.33) .. (462.11,146.8) -- (488.75,135.3) -- cycle ; \draw   (501.35,110.31) .. controls (503.99,112.09) and (506.56,114.16) .. (508.98,116.48) .. controls (524.81,131.68) and (528.59,152.43) .. (517.42,162.82) .. controls (506.25,173.21) and (484.35,169.31) .. (468.52,154.11) .. controls (466.1,151.78) and (463.95,149.33) .. (462.11,146.8) ;

        % Text Node
        \draw (148.91,106.58) node [anchor=north west][inner sep=0.75pt]   [align=left] {$\displaystyle k+1$};
        % Text Node
        \draw (226.13,168.54) node [anchor=north west][inner sep=0.75pt]   [align=left] {$\displaystyle k+2$};
        % Text Node
        \draw (291.77,122.41) node [anchor=north west][inner sep=0.75pt]   [align=left] {$\displaystyle k+\ell -1$};
        % Text Node
        \draw (284.62,161.81) node [anchor=north west][inner sep=0.75pt]   [align=left] {...};
        % Text Node
        \draw (279.7,102.06) node [anchor=north west][inner sep=0.75pt]   [align=left] {$\displaystyle k+\ell $};
        % Text Node
        \draw (234.42,112.76) node [anchor=north west][inner sep=0.75pt]   [align=left] {$\displaystyle j$};
        % Text Node
        \draw (173.57,123.37) node [anchor=north west][inner sep=0.75pt]   [align=left] {$\displaystyle j-1$};
        % Text Node
        \draw (197.78,151.53) node [anchor=north west][inner sep=0.75pt]   [align=left] {$\displaystyle j-2$};
        % Text Node
        \draw (136.53,87.25) node [anchor=north west][inner sep=0.75pt]   [align=left] {$\displaystyle j+1$};
        % Text Node
        \draw (134.48,62.29) node [anchor=north west][inner sep=0.75pt]   [align=left] {$\displaystyle j+2$};
        % Text Node
        \draw (165.59,34.93) node [anchor=north west][inner sep=0.75pt]   [align=left] {...};
        % Text Node
        \draw (236.77,76.61) node [anchor=north west][inner sep=0.75pt]   [align=left] {$\displaystyle j-3$};
        % Text Node
        \draw (233.15,57.54) node [anchor=north west][inner sep=0.75pt]   [align=left] {$\displaystyle j-4$};
        % Text Node
        \draw (250.02,11.32) node [anchor=north west][inner sep=0.75pt]   [align=left] {($\displaystyle \star $)};
        % Text Node
        \draw (386.51,107.12) node [anchor=north west][inner sep=0.75pt]   [align=left] {$\displaystyle j+3$};
        % Text Node
        \draw (460.62,178.6) node [anchor=north west][inner sep=0.75pt]   [align=left] {$\displaystyle k+2$};
        % Text Node
        \draw (523.3,127.37) node [anchor=north west][inner sep=0.75pt]   [align=left] {$\displaystyle k+\ell -1$};
        % Text Node
        \draw (520.71,160.13) node [anchor=north west][inner sep=0.75pt]   [align=left] {...};
        % Text Node
        \draw (513.74,102.82) node [anchor=north west][inner sep=0.75pt]   [align=left] {$\displaystyle k+\ell $};
        % Text Node
        \draw (456.89,106.95) node [anchor=north west][inner sep=0.75pt]   [align=left] {$\displaystyle j+1$};
        % Text Node
        \draw (412.55,126.7) node [anchor=north west][inner sep=0.75pt]   [align=left] {$\displaystyle j+2$};
        % Text Node
        \draw (430.69,157.64) node [anchor=north west][inner sep=0.75pt]   [align=left] {$\displaystyle k+1$};
        % Text Node
        \draw (373.69,83.59) node [anchor=north west][inner sep=0.75pt]   [align=left] {$\displaystyle j+4$};
        % Text Node
        \draw (400.93,35.39) node [anchor=north west][inner sep=0.75pt]   [align=left] {...};
        % Text Node
        \draw (472.64,79.22) node [anchor=north west][inner sep=0.75pt]   [align=left] {$\displaystyle j-1$};
        % Text Node
        \draw (490.82,70.08) node [anchor=north west][inner sep=0.75pt]   [align=left] {$\displaystyle j$};
        % Text Node
        \draw (270.17,15.07) node [anchor=north west][inner sep=0.75pt]    {$\left( \sigma _{2}^{-1} \sigma _{1} \sigma _{2} \sigma _{1}\right)^{\frac{k-2}{2}}$};
        % Text Node
        \draw (462.06,48.53) node [anchor=north west][inner sep=0.75pt]   [align=left] {$\displaystyle j-2$};
        % Text Node
        \draw (440.38,29.21) node [anchor=north west][inner sep=0.75pt]   [align=left] {$\displaystyle j-3$};
        % Text Node
        \draw (55,26.4) node [anchor=north west][inner sep=0.75pt]    {};

    \end{tikzpicture}
    \captionof{figure}{Run $(\sigma_2^{-1}\sigma_1\sigma_2\sigma_1)^{\frac{k-2}{2}}$ on $(\star)$}

    Since the zeroth position is always occupied by the $j+1$, we need to $\sigma_1^{j+1}$ to get it back to it's original position. Note, using this notation, if $j$ or $j-1$ was zero, then we put the edge that was switched with the zero into the zeroth position. Thus, we have three different ending states. In the bottom two diagrams in [FIGURE] below is the case when $j$ or $j-1$ is $0$.

    \begin{tikzpicture}[x=0.75pt,y=0.75pt,yscale=-1,xscale=1]
        %uncomment if require: \path (0,636); %set diagram left start at 0, and has height of 636

        %Shape: Ellipse [id:dp7533921208277927] 
        \draw   (137.54,79.67) .. controls (137.54,42.18) and (167.07,11.79) .. (203.5,11.79) .. controls (239.92,11.79) and (269.45,42.18) .. (269.45,79.67) .. controls (269.45,117.15) and (239.92,147.55) .. (203.5,147.55) .. controls (167.07,147.55) and (137.54,117.15) .. (137.54,79.67) -- cycle ;
        %Shape: Arc [id:dp2882722909305362] 
        \draw  [draw opacity=0] (263.52,107.19) .. controls (265.95,108.98) and (268.31,111.04) .. (270.53,113.37) .. controls (285.06,128.56) and (288.53,149.31) .. (278.28,159.7) .. controls (268.02,170.09) and (247.93,166.2) .. (233.39,151) .. controls (231.17,148.67) and (229.2,146.22) .. (227.51,143.69) -- (251.96,132.18) -- cycle ; \draw   (263.52,107.19) .. controls (265.95,108.98) and (268.31,111.04) .. (270.53,113.37) .. controls (285.06,128.56) and (288.53,149.31) .. (278.28,159.7) .. controls (268.02,170.09) and (247.93,166.2) .. (233.39,151) .. controls (231.17,148.67) and (229.2,146.22) .. (227.51,143.69) ;
        %Straight Lines [id:da7703748103912469] 
        \draw    (273.59,82.04) -- (328.6,81.49) ;
        \draw [shift={(330.6,81.47)}, rotate = 179.42] [color={rgb, 255:red, 0; green, 0; blue, 0 }  ][line width=0.75]    (10.93,-3.29) .. controls (6.95,-1.4) and (3.31,-0.3) .. (0,0) .. controls (3.31,0.3) and (6.95,1.4) .. (10.93,3.29)   ;
        %Straight Lines [id:da40299510016046347] 
        \draw    (173.79,147.16) -- (173.13,191.1) ;
        \draw [shift={(173.1,193.1)}, rotate = 270.87] [color={rgb, 255:red, 0; green, 0; blue, 0 }  ][line width=0.75]    (10.93,-3.29) .. controls (6.95,-1.4) and (3.31,-0.3) .. (0,0) .. controls (3.31,0.3) and (6.95,1.4) .. (10.93,3.29)   ;
        %Shape: Ellipse [id:dp7504591206154643] 
        \draw   (343.13,76.63) .. controls (343.13,39.14) and (372.66,8.75) .. (409.09,8.75) .. controls (445.51,8.75) and (475.04,39.14) .. (475.04,76.63) .. controls (475.04,114.12) and (445.51,144.51) .. (409.09,144.51) .. controls (372.66,144.51) and (343.13,114.12) .. (343.13,76.63) -- cycle ;
        %Shape: Arc [id:dp3607703426088966] 
        \draw  [draw opacity=0] (469.11,104.16) .. controls (471.54,105.94) and (473.9,108) .. (476.12,110.33) .. controls (490.65,125.53) and (494.12,146.27) .. (483.87,156.67) .. controls (473.61,167.06) and (453.52,163.16) .. (438.98,147.96) .. controls (436.76,145.64) and (434.79,143.18) .. (433.1,140.65) -- (457.55,129.15) -- cycle ; \draw   (469.11,104.16) .. controls (471.54,105.94) and (473.9,108) .. (476.12,110.33) .. controls (490.65,125.53) and (494.12,146.27) .. (483.87,156.67) .. controls (473.61,167.06) and (453.52,163.16) .. (438.98,147.96) .. controls (436.76,145.64) and (434.79,143.18) .. (433.1,140.65) ;
        %Shape: Ellipse [id:dp00872404803263449] 
        \draw   (127.79,263.4) .. controls (127.79,225.91) and (157.31,195.52) .. (193.74,195.52) .. controls (230.16,195.52) and (259.69,225.91) .. (259.69,263.4) .. controls (259.69,300.88) and (230.16,331.27) .. (193.74,331.27) .. controls (157.31,331.27) and (127.79,300.88) .. (127.79,263.4) -- cycle ;
        %Shape: Arc [id:dp5672515404140084] 
        \draw  [draw opacity=0] (253.76,290.92) .. controls (256.19,292.71) and (258.55,294.77) .. (260.77,297.09) .. controls (275.31,312.29) and (278.77,333.04) .. (268.52,343.43) .. controls (258.26,353.82) and (238.17,349.93) .. (223.63,334.73) .. controls (221.41,332.4) and (219.45,329.95) .. (217.75,327.42) -- (242.2,315.91) -- cycle ; \draw   (253.76,290.92) .. controls (256.19,292.71) and (258.55,294.77) .. (260.77,297.09) .. controls (275.31,312.29) and (278.77,333.04) .. (268.52,343.43) .. controls (258.26,353.82) and (238.17,349.93) .. (223.63,334.73) .. controls (221.41,332.4) and (219.45,329.95) .. (217.75,327.42) ;
        %Shape: Ellipse [id:dp43665981142202503] 
        \draw   (327.8,246.94) .. controls (327.8,209.45) and (357.33,179.06) .. (393.75,179.06) .. controls (430.18,179.06) and (459.71,209.45) .. (459.71,246.94) .. controls (459.71,284.43) and (430.18,314.82) .. (393.75,314.82) .. controls (357.33,314.82) and (327.8,284.43) .. (327.8,246.94) -- cycle ;
        %Shape: Arc [id:dp6711321273772175] 
        \draw  [draw opacity=0] (453.78,274.46) .. controls (456.21,276.25) and (458.56,278.31) .. (460.79,280.64) .. controls (475.32,295.83) and (478.79,316.58) .. (468.53,326.97) .. controls (458.28,337.36) and (438.18,333.47) .. (423.65,318.27) .. controls (421.43,315.94) and (419.46,313.49) .. (417.76,310.96) -- (442.22,299.45) -- cycle ; \draw   (453.78,274.46) .. controls (456.21,276.25) and (458.56,278.31) .. (460.79,280.64) .. controls (475.32,295.83) and (478.79,316.58) .. (468.53,326.97) .. controls (458.28,337.36) and (438.18,333.47) .. (423.65,318.27) .. controls (421.43,315.94) and (419.46,313.49) .. (417.76,310.96) ;
        %Straight Lines [id:da5196639602353341] 
        \draw    (297.15,173.78) -- (325.72,203.12) ;
        \draw [shift={(327.11,204.55)}, rotate = 225.76] [color={rgb, 255:red, 0; green, 0; blue, 0 }  ][line width=0.75]    (10.93,-3.29) .. controls (6.95,-1.4) and (3.31,-0.3) .. (0,0) .. controls (3.31,0.3) and (6.95,1.4) .. (10.93,3.29)   ;

        % Text Node
        \draw (149.75,95.9) node [anchor=north west][inner sep=0.75pt]   [align=left] {$\displaystyle j+3$};
        % Text Node
        \draw (226.42,168.12) node [anchor=north west][inner sep=0.75pt]   [align=left] {$\displaystyle k+2$};
        % Text Node
        \draw (284.82,121.95) node [anchor=north west][inner sep=0.75pt]   [align=left] {$\displaystyle k+\ell -1$};
        % Text Node
        \draw (280.03,157.73) node [anchor=north west][inner sep=0.75pt]   [align=left] {...};
        % Text Node
        \draw (276.01,94.76) node [anchor=north west][inner sep=0.75pt]   [align=left] {$\displaystyle k+\ell $};
        % Text Node
        \draw (216.22,103.26) node [anchor=north west][inner sep=0.75pt]   [align=left] {$\displaystyle j+1$};
        % Text Node
        \draw (169.17,116.79) node [anchor=north west][inner sep=0.75pt]   [align=left] {$\displaystyle j+2$};
        % Text Node
        \draw (198.47,150.26) node [anchor=north west][inner sep=0.75pt]   [align=left] {$\displaystyle k+1$};
        % Text Node
        \draw (144.65,71.52) node [anchor=north west][inner sep=0.75pt]   [align=left] {$\displaystyle j+4$};
        % Text Node
        \draw (170.76,32.28) node [anchor=north west][inner sep=0.75pt]   [align=left] {...};
        % Text Node
        \draw (232.68,76.58) node [anchor=north west][inner sep=0.75pt]   [align=left] {$\displaystyle j-1$};
        % Text Node
        \draw (248.41,58.97) node [anchor=north west][inner sep=0.75pt]   [align=left] {$\displaystyle j$};
        % Text Node
        \draw (218.58,42.57) node [anchor=north west][inner sep=0.75pt]   [align=left] {$\displaystyle j-2$};
        % Text Node
        \draw (202.52,23.25) node [anchor=north west][inner sep=0.75pt]   [align=left] {$\displaystyle j-3$};
        % Text Node
        \draw (287.71,54.7) node [anchor=north west][inner sep=0.75pt]    {$\sigma _{1}^{j+1}$};
        % Text Node
        \draw (140.66,158.46) node [anchor=north west][inner sep=0.75pt]    {$\sigma _{1}^{j+1}$};
        % Text Node
        \draw (442.55,108.97) node [anchor=north west][inner sep=0.75pt]   [align=left] {$\displaystyle 0$};
        % Text Node
        \draw (414.05,124.76) node [anchor=north west][inner sep=0.75pt]   [align=left] {$\displaystyle 1$};
        % Text Node
        \draw (393.07,123.61) node [anchor=north west][inner sep=0.75pt]   [align=left] {$\displaystyle 2$};
        % Text Node
        \draw (346.38,77.19) node [anchor=north west][inner sep=0.75pt]   [align=left] {...};
        % Text Node
        \draw (370.31,27.01) node [anchor=north west][inner sep=0.75pt]   [align=left] {$\displaystyle j-1$};
        % Text Node
        \draw (357.42,44.82) node [anchor=north west][inner sep=0.75pt]   [align=left] {$\displaystyle j$};
        % Text Node
        \draw (422.24,17.52) node [anchor=north west][inner sep=0.75pt]   [align=left] {...};
        % Text Node
        \draw (429.97,57.73) node [anchor=north west][inner sep=0.75pt]   [align=left] {$\displaystyle k-1$};
        % Text Node
        \draw (458.66,83.21) node [anchor=north west][inner sep=0.75pt]   [align=left] {$\displaystyle k$};
        % Text Node
        \draw (404.61,148.77) node [anchor=north west][inner sep=0.75pt]   [align=left] {$\displaystyle k+1$};
        % Text Node
        \draw (436.89,165.51) node [anchor=north west][inner sep=0.75pt]   [align=left] {$\displaystyle k+2$};
        % Text Node
        \draw (493.5,115.2) node [anchor=north west][inner sep=0.75pt]   [align=left] {$\displaystyle k+\ell -1$};
        % Text Node
        \draw (487.01,155.41) node [anchor=north west][inner sep=0.75pt]   [align=left] {...};
        % Text Node
        \draw (480.6,92.01) node [anchor=north west][inner sep=0.75pt]   [align=left] {$\displaystyle k+\ell $};
        % Text Node
        \draw (202.81,307.18) node [anchor=north west][inner sep=0.75pt]   [align=left] {$\displaystyle 0$};
        % Text Node
        \draw (224.49,295.06) node [anchor=north west][inner sep=0.75pt]   [align=left] {$\displaystyle 1$};
        % Text Node
        \draw (177.72,310.37) node [anchor=north west][inner sep=0.75pt]   [align=left] {$\displaystyle 2$};
        % Text Node
        \draw (151.83,217.16) node [anchor=north west][inner sep=0.75pt]   [align=left] {...};
        % Text Node
        \draw (220.62,240.49) node [anchor=north west][inner sep=0.75pt]   [align=left] {$\displaystyle k-1$};
        % Text Node
        \draw (239.31,265.97) node [anchor=north west][inner sep=0.75pt]   [align=left] {$\displaystyle k$};
        % Text Node
        \draw (191.26,337.54) node [anchor=north west][inner sep=0.75pt]   [align=left] {$\displaystyle k+1$};
        % Text Node
        \draw (218.54,353.27) node [anchor=north west][inner sep=0.75pt]   [align=left] {$\displaystyle k+2$};
        % Text Node
        \draw (277.15,305.97) node [anchor=north west][inner sep=0.75pt]   [align=left] {$\displaystyle k+\ell -1$};
        % Text Node
        \draw (274.46,342.89) node [anchor=north west][inner sep=0.75pt]   [align=left] {...};
        % Text Node
        \draw (262.25,280.77) node [anchor=north west][inner sep=0.75pt]   [align=left] {$\displaystyle k+\ell $};
        % Text Node
        \draw (437.67,251.37) node [anchor=north west][inner sep=0.75pt]   [align=left] {$\displaystyle 0$};
        % Text Node
        \draw (398.72,295.06) node [anchor=north west][inner sep=0.75pt]   [align=left] {$\displaystyle 1$};
        % Text Node
        \draw (377.73,293.91) node [anchor=north west][inner sep=0.75pt]   [align=left] {$\displaystyle 2$};
        % Text Node
        \draw (354.05,201.69) node [anchor=north west][inner sep=0.75pt]   [align=left] {...};
        % Text Node
        \draw (418.64,223.03) node [anchor=north west][inner sep=0.75pt]   [align=left] {$\displaystyle k-1$};
        % Text Node
        \draw (425.21,276.41) node [anchor=north west][inner sep=0.75pt]   [align=left] {$\displaystyle k$};
        % Text Node
        \draw (390.27,321.08) node [anchor=north west][inner sep=0.75pt]   [align=left] {$\displaystyle k+1$};
        % Text Node
        \draw (422.56,337.81) node [anchor=north west][inner sep=0.75pt]   [align=left] {$\displaystyle k+2$};
        % Text Node
        \draw (478.17,290.51) node [anchor=north west][inner sep=0.75pt]   [align=left] {$\displaystyle k+\ell -1$};
        % Text Node
        \draw (470.29,327.15) node [anchor=north west][inner sep=0.75pt]   [align=left] {...};
        % Text Node
        \draw (467.27,259.32) node [anchor=north west][inner sep=0.75pt]   [align=left] {$\displaystyle k+\ell $};
        % Text Node
        \draw (314.89,167.05) node [anchor=north west][inner sep=0.75pt]    {$\sigma _{1}^{j+1}$};
        % Text Node
        \draw (55,26.4) node [anchor=north west][inner sep=0.75pt]    {};

    \end{tikzpicture}

    \captionof{figure}{Result}

    Let $\psi_j = \sigma_1^{j+1}(\sigma_2^{-1}\sigma_1\sigma_2\sigma_1)^{\frac{k-2}{2}}\sigma_2^{-1}\sigma_1^2\sigma_2\sigma_1^{k+1-j}$ denote the algorithm to transpose the edges $j-1$ and $j$. To switch any edges $a$ and $b$ from cycle $C_1$ where $b < a$, it is the conjugate
    \[ \Psi_{(a,b)} = (\psi_{b+2}\psi_{b+3}\dots\psi_{a-1}\psi_{a})^{-1}\psi_{b+1}(\psi_{b+2}\psi_{b+3}\dots\psi_{a-1}\psi_{a}). \]
    To switch any edge $a$ from cycle $C_1$ and $b$ from cycle $C_2$ is the conjugate
    \[ (\sigma_2^{k-b})^{-1}\Psi_{(a, 0)}\sigma_2^{k-b} \]

    \textbf{Inductive Step:} Let $C_i$ be the cycle sharing the edge $\ell$ with the new cycle $C_{n+1}$ and let $\sigma_i$ and $\sigma_{n+1}$ be the cycle in $C_i$ and $C_{n+1}$ respectively. Let $k$ be the size of $C_{n+1}$. Let $\psi_{\ell + 1}$ be the algorithm to swap the edges $\ell$ and $\ell + 1$ in $C_i$.

    We claim that if $\ell \leq j - 1 < j \leq k$ then the product,
    \[ (\sigma_{n+1}\sigma_i\sigma_{n+1}^{k+1-j})^{-1}\psi_{\ell + 1}(\sigma_{n+1}\sigma_i\sigma_{n+1}^{k+1-j}) \]
    will transpose the edges $j-1$ and $j$.

    Following the same logic from above, we can switch any edges $a$ and $b$ in $C_{n+1}$ where $b < a$ by the conjugate
    \[ \Phi_{(a,b)}=(\phi_{b+2}\phi_{b+3}\dots\phi_{a-1}\phi_{a})^{-1}\phi_{b+1}(\phi_{b+2}\phi_{b+3}\dots\phi_{a-1}\phi_{a}) \]
    where $\phi_j = (\sigma_{n+1}\sigma_i\sigma_{n+1}^{k+1-j})^{-1}\psi_{\ell + 1}(\sigma_{n+1}\sigma_i\sigma_{n+1}^{k+1-j})$.
\end{proof}

\begin{theorem}
    Suppose that $G$ is a 2d-Rubik's shape with distinguished cycles $C_1, C_2$ and that the length of both $C_1, C_2$ are odd. Then $G$ is not a complete 2d-Rubik's shape.
\end{theorem}

\begin{proof}
    Since neither of $C_1$ or $C_2$ has even length, the group of permutations corresponding to the Rubik's shape G is generated by $\sigma_1$ and $\sigma_2$ which are both even permutations, and so must be a subgroup of $A_{2k+2\ell}$, and thus a proper subgroup of $S_{2k+2\ell}$.

    % Label the edges of $C_1$ by $(0, 1, ..., 2k)$ and the edges of $C_2$ by $(0, 2k + (2\ell - 1), 2k + (2\ell - 2), \dots, k + 2, k+ 1)$, so that if we identify $\sigma_1, \sigma_2$ with the corresponding elements of Sym$(2k + 2\ell)$ then $\sigma_1$ is the cycle $(0, 1, ..., 2k)$ and $\sigma_2$ is the cycle $(0, 2k + (2\ell - 1), 2k + (2\ell - 2), \dots, k + 2, k+ 1)$. Then
    % \[ \sigma_1 = (0, 2k - 1, 2k)(0, 2k - 3, 2k - 2)\dots(0, 3, 4)(0, 1, 2), \]
    % \[ \sigma_2 = (0, k + 2, k + 1)(0, k + 4, k + 3)\dots(0, 2k + (2\ell - 4), 2k + (2\ell - 3))(0, 2k + (2\ell - 2), 2k + (2\ell - 1)). \]
    % Since $\sigma_1$ and $\sigma_2$ are generated by 3-cycles, then $\langle\sigma_1, \sigma_2\rangle$ is generated by the set of 3-cycles. Thus, $\langle\sigma_1, \sigma_2\rangle \subseteq A_{2k + 2\ell}$ and $G$ is not complete.
\end{proof}

% \section{Open questions}
% We only investigated the $2 \times 2$ Rubik's square and naturally the question about if all configurations of an $n \times n$ Rubik's square is complete comes up. Secondly, it is known that the optimal solution to any Rubik's cube is $20$ and called God's number and that our algorithm is by brute force and not optimal. Is there a God's number for the Rubik's square and is it the same for different sizes? The obvious question for the Rubik's shape is if we share more than one side, is it still complete and in those cases, do we still need at least one piece having even length?

% \section*{Acknowledgment}
% The author would like to thank the Math Department of the University of North Texas’s Incubator project for its encouragement and guidance.

\bibliography{refs}{}
\bibliographystyle{plain}

\end{document}